\documentclass[11pt,reqno]{amsart}
\usepackage{amsmath,mathrsfs}

\usepackage{enumerate}
\usepackage[english]{babel}
\usepackage{hyperref}
\usepackage[all,cmtip]{xy}
\entrymodifiers={+!!<0pt,\fontdimen22\textfont2>}

\usepackage[a4paper,text={15.5cm,25.2cm},centering]{geometry}
\theoremstyle{theorem}
\newtheorem{thm}{Theorem}[section]
\newtheorem{lem}[thm]{Lemma}
\newtheorem{prop}[thm]{Proposition}
\newtheorem{cor}[thm]{Corollary}

\makeatletter
\newtheorem*{rep@theorem}{\rep@title}
\newcommand{\newreptheorem}[2]{%
\newenvironment{rep#1}[1]{%
 \def\rep@title{#2 \ref{##1}}%
 \begin{rep@theorem}}%
 {\end{rep@theorem}}}
\makeatother

\newreptheorem{thm}{Theorem}

\theoremstyle{definition}
\newtheorem{defn}[thm]{Definition}
\newtheorem{ex}[thm]{Example}

\newtheorem{ques}[thm]{Question}
\theoremstyle{remark}
\newtheorem{rem}[thm]{Remark}
\newtheorem*{ack}{Acknowledgement}

\def\Z{{\mathbb Z}} 

\def\B{{\mathcal L}} 
\def\K{{\mathcal K}} 

\mathchardef\ordinarycolon\mathcode`\: 
\def\vcentcolon{\mathrel{\mathop\ordinarycolon}} 
\providecommand*\coloneqq{\mathrel{\vcentcolon\mkern-1.2mu}=}

\def\cast{$C^{*}$}
\DeclareMathOperator\id{id} 

\def\B{\mathcal B}
\def\Cl{C^{*}_{\lambda}}
\def\Cu{C^{*}_{u}}

\def\G{{\Gamma}}
\def\L{L}

\def\Ad{\mathrm{Ad}}

\begin{document}

\title{On the invariant uniform Roe algebra}
\author{Takeshi Katsura and Otgonbayar Uuye}
\date{\today}   

\address{
Takeshi Katsura\\
Department of Mathematics\\ 
Faculty of Science and Technology\\ 
Keio University\\ 
3-14-1 Hiyoshi, Kouhoku-ku\\ 
Yokohama 223-8522\\
JAPAN}
\address{
Otgonbayar Uuye\\
School of Mathematics\\
Cardiff University\\
Senghennydd Road\\
Cardiff, Wales, UK\\
CF24 4AG}

\begin{abstract} Let $\G$ be a countable discrete group. We show that $\G$ has the approximation property if and only if $\G$ is exact and for any operator space $S \subseteq \K(H)$ we have
	\begin{equation*}
	\Cu(\G)^{\G} \otimes S = (\Cu(\G) \otimes S)^{\G},
	\end{equation*}
where $\Cu(\G)$ is the uniform Roe algebra with the right adjoint $\G$-action. This answers a question of J.\ Zacharias. We also show that characterisations of several properties of $\G$ in terms of the reduced group \cast-algebra $\Cl(\G)$ apply to the invariant uniform Roe algebra $\Cu(\G)^{\G}$. 
\end{abstract}

\maketitle
\section{Introduction}

Let $\G$ be a countable discrete group and let $l^{2}\G$ denote the Hilbert space of square-summable functions on $\G$ with orthonormal basis $\left\{\delta_{s} \mid s \in \G\right\}$. We identify operators $T \in \B(l^{2}\G)$ with the $(\G \times \G)$-matrix  given by $[\langle T\delta_{s}, \delta_{t}\rangle]_{s, t\in \G}$.

Let $\lambda$ and $\rho$ denote the left and right regular representations of $\G$ on $l^{2}\G$ given by 
	\begin{equation*}
	\lambda(g)\delta_{s} \coloneqq \delta_{gs}\quad \text{and}\quad \rho(g)\delta_{s} \coloneqq \delta_{sg^{-1}},\quad g, s \in \G,
	\end{equation*}
respectively. The {\em group \cast-algebra} $\Cl(\G)$ of $\G$ is the \cast-algebra generated by $\lambda(\G) \subseteq \B(l^{2}\G)$.

In this paper, we consider two \cast-subalgebras of $\B(l^{2}\G)$ containing $\Cl(\G)$. The first is the {\em uniform Roe algebra}  $\Cu(\G)$. This algebra may be thought of as the closure of scalar $(\G \times \G)$-matrices $[\alpha_{s, t}]$ of finite width (i.e. $\{st^{-1} \in \G \mid \alpha_{s,t} \neq 0\}$ is finite) with uniformly bounded entries acting on $l^{2}\Gamma$. The second is the {\em group von Neumann algebra} $\L(\G) \coloneqq \Cl(\G)''$.  By a fundamental result of Murray and von Neumann (c.f.\ \cite[Secion 6.1]{MR2391387}), we have $\L(\G) = \rho(\G)'$, the commutant of $\rho(\G)$. It follows that $\L(\G)$ can be identified with the algebra of elements in $\B(\l^{2}\G)$ that are constant down the diagonals i.e.\ satisfying $\alpha_{s,t} = \alpha_{st^{-1}, e}$ for $s$, $t \in \G$. 

The adjoint action $\Ad(\rho)$ of $\G$ on $\B(l^{2}\G)$ preserves $\Cu(\G)$ and the invariant elements are precisely the intersection:
\begin{equation*}
\Cu(\G)^{\G} = \Cu(\G) \cap \L(\G).
\end{equation*}

Clearly, we have
	\begin{equation*}
	\Cl(\G) \subseteq \Cu(\G)^{\G}.
	\end{equation*}
\begin{defn}[{cf.\ \cite[11.5.3]{MR2007488}}]
We say that $\G$ has the {\em invariant translation approximation property} (ITAP) if 
	\begin{equation*}
	\Cl(\G) = \Cu(\G)^{\G}.
	\end{equation*}
\end{defn}

Conjecturally, all countable discrete groups have the ITAP. It is proved in \cite{MR2007488} that {\em amenable} groups and {\em finitely generated free} groups have the ITAP. 

Haagerup and Kraus defined and studied the {\em approximation property} (AP) for a group, which we recall below briefly for the convenience of the reader. See \cite{MR1220905, MR2391387} for more details.

Let $A(\G) = A_{2}(\G)$ denote the Fourier algebra of $\G$ and let $B_{2}(\G) = M_{0}A(\G)$ denote the algebra of completely bounded multipliers of $A(\G)$. Let $Q(\G)$ denote the completion of $l^{1}\G$ in the norm
	\begin{equation*}
	||f||_{Q} \coloneqq \sup \left\{\left|\sum f(s)u(s)\right| \mid u \in B_{2}(\G), ||u||_{B_{2}} \le 1\right\}.
	\end{equation*}

Then $B_{2}(\G) \cong Q(\G)^{*}$.  
\begin{defn}[{Haagerup-Kraus, \cite[Definition 1.1]{MR1220905}}]
We say that $\G$ has the {\em approximation property} (AP) if the constant function $1 \in B_{2}(\G)$ is in the $\sigma(B_{2}(\G), Q(\G))$-closure of $A(\G)$. 
\end{defn}

\begin{ex}\label{ex ap}
\begin{enumerate}[(i)]
\item Weakly amenable, in particular amenable, groups have the AP (cf.\  \cite{MR1220905}). 
\item $SL_{3}(\Z)$ does not have the AP (cf.\ \cite{MR2838352}). More generally, lattices in  connected simple Lie groups with finite center and real rank $\ge 2$ do not have the AP (cf.\ \cite{Haagerup:2012fk}). 
\end{enumerate}
\end{ex}

In \cite{MR2215118}, Zacharias proved that groups with the AP have the ITAP. In fact, he proved the following stronger statement. We fix an infinite dimensional separable Hilbert space $H$. The \cast-algebra of all bounded (respectively, compact) linear operators on $H$ is denoted by $\B(H)$ (respectively, $\K(H)$).

Let $S \subseteq \B(H)$ be a (concrete) operator space. 
Clearly, we have inclusions
	\begin{equation*}
	\Cl(\G) \otimes S \subseteq \Cu(\G)^{\G} \otimes S \subseteq (\Cu(\G) \otimes S)^{\G},
	\end{equation*}
where $\G$ is acting trivially on $S$. 

Recall that a countable discrete group $\G$ said to be {\em exact} if $\Cl(\G)$ is exact (cf.\ \cite{MR1721796,MR2391387}).
\begin{thm}[{Zacharias \cite[Theorem 3.2]{MR2215118}}]\label{thm Zac} Let $\G$ be a countable discrete exact group. Then the following are equivalent.
\begin{enumerate}[(i)]
\item  The group $\G$ has the AP.
\item The equality $\Cl(\G) \otimes S = (\Cu(\G) \otimes S)^{\G}$ holds for all operator spaces $S \subseteq \K(H)$.
\item The equality $\Cl(\G) \otimes S = (\Cu(\G) \otimes S)^{\G}$ holds for all operator spaces $S \subseteq \B(H)$.
\end{enumerate}
\end{thm}

\begin{rem}
\begin{enumerate}[(i)]
\item Groups with the AP are exact by \cite[Theorem 2.1]{MR1220905} and the discussion following it.
\item Group $\G$ has the ITAP if and only if the identity $\Cl(\G) \otimes S = \Cu(\G)^{\G} \otimes S$ holds for all operator spaces $S \subseteq \K(H)$ if and only if the identity $\Cl(\G) \otimes S = \Cu(\G)^{\G} \otimes S$ holds for all operator spaces $S \subseteq \B(H)$. This is trivially true.
\end{enumerate}
\end{rem}
At the end of \cite{MR2215118}, Zacharias asked the following: 

\begin{ques}\label{ques Z} Is it true that if $\G$ is a discrete group, then 
	\begin{equation*}
	\Cu(\G)^{\G} \otimes S = (\Cu(\G) \otimes S)^{\G}
	\end{equation*} 
for all operator spaces  $S \subseteq \K(H)$?
\end{ques}

Our main result is the following, which answers Question~\ref{ques Z} in the negative.
\begin{thm}\label{thm main} Let $\G$ be a countable discrete exact group. Then the following  are equivalent.
\begin{enumerate}
\item\label{it AP} The group $\G$ has the AP.
\item\label{it Slice} The equality $\Cu(\G)^{\G} \otimes S = (\Cu(\G) \otimes S)^{\G}$ holds for all operator spaces $S \subseteq \K(H)$.
\item The equality $\Cu(\G)^{\G} \otimes S = (\Cu(\G) \otimes S)^{\G}$ holds for all operator spaces $S \subseteq \B(H)$.
\end{enumerate}
\end{thm} 

Since $SL_{3}(\Z)$ does not have the AP by \cite{MR2838352}, we get the following as a corollary.
\begin{cor} Let $\G = SL_{3}(\Z)$. Then there exists an operator space $S \subseteq \K(H)$ such that
	\begin{equation*}
	\Cu(\G)^{\G} \otimes S \neq (\Cu(\G) \otimes S)^{\G}.
	\end{equation*}
\qed	
\end{cor}

We also show that various approximation properties hold for $\Cl(\G)$ if and only if they hold for $\Cu(\G)^{\G}$. See Theorem~\ref{thm app} for the precise statement.

\begin{ack} The authors were partially supported by the Danish Research Council through the Centre for Symmetry and Deformation at the University of Copenhagen. The first author was also partially supported
by the Japan Society for the Promotion of Science.
The second author is an EPSRC fellow.
\end{ack}

\section{Proof of the Main Theorem (\ref{thm main})}
\subsection{The Slice Map Property}\label{sec SMP} 
We start with some definitions.
All tensor products we consider are spatial (minimal) ones. 

\begin{defn}[{cf.\ \cite{MR0410402}}] Let $T \subseteq A$ and $S \subseteq B$ be operator spaces. The {\em Fubini product} $F(T, S, A \otimes B)$ of $T$ and $S$ 
is defined as the set of all $x \in A \otimes B$ such that $(\phi \otimes \id_{B})(x) \in S$ for all $\phi \in A^{*}$ and $(\id_{A} \otimes \psi)(x) \in T$ for all $\psi \in B^{*}$. 

\end{defn}
\begin{rem} We always have 
	\begin{equation}\label{eq F}
	T \otimes S \subseteq F(T, S, A \otimes B) = F(T, B, A \otimes B) \cap F(A, S, A \otimes B) \subseteq A \otimes B.
	\end{equation}
\end{rem}

\begin{defn}
We say that $(T, S, A \otimes B)$ has the {\em slice map property} if 
	\begin{equation*}
	F(T, S, A \otimes B) = T \otimes S.
	\end{equation*}
We say that $A$ has the slice map property for $B$ if $(A, S, A \otimes B)$ has the slice map property for all operator spaces $S \subseteq B$.
\end{defn}

The following is a simple but key result that connects fixed points to the slice map property.

\begin{prop} Let $\G$ be a countable discrete group and let $S$ and $T$ be operator spaces. Suppose that $\G$ is acting on $T$ by completely bounded maps. Letting $\G$ act trivially on $S$, we get a $\G$-action on $T \otimes S$. Then we have
	\begin{equation*}
	F(T^{\G}, S, T \otimes S) = (T \otimes S)^{\G}.
	\end{equation*}
\end{prop}
\begin{proof} Let $x \in T \otimes S$. Then for any $\psi \in S^{*}$ and $g \in \G$, we have
	\begin{equation}\label{eq square}
	g((\id_{T} \otimes \psi)(x)) = (\id_{T} \otimes \psi) (g \otimes \id_{S})(x).	
	\end{equation}
Hence if $x \in (T \otimes S)^{\G}$, then $(\id_{T} \otimes \psi)(x) \in T^{\G}$ for all $\psi \in S^{*}$. It follows that $(T \otimes S)^{\G} \subseteq F(T^{\G}, S, T \otimes S)$. Conversely, if $x \in F(T^{\G}, S, T \otimes S)$, then for any $\psi \in S^{*}$ and $g \in \G$, we have $(\id_{T} \otimes \psi)(x - (g \otimes \id_{S})(x)) = 0$ by (\ref{eq square}). Since $\id_{T} \otimes S^{*}$ separates points of $T \otimes S$, we see that $x \in (T \otimes S)^{\G}$.
\end{proof}

\begin{cor}\label{cor cug} Let $S \subseteq \B(H)$ be an operator space. Then $\Cu(\G)^{\G} \otimes S = (\Cu(\G) \otimes S)^{\G}$ if and only if $(\Cu(\G)^{\G}, S, \Cu(\G) \otimes S)$ has the slice map property.\qed
\end{cor}

\begin{lem}\label{lem aotimes} Let $T \subseteq A$ and $S \subseteq B$ be operator spaces. Suppose that $(A, S, A \otimes B)$ has the slice map property. Then
	\begin{equation*}
	F(T, S, A \otimes S) = F(T, S, A \otimes B)
	\end{equation*}
\end{lem}
\begin{proof} By assumption, we have
	\begin{equation*}
	A \otimes S = F(A, S, A \otimes B).
	\end{equation*}
On the other hand, any $\psi \in S^{*}$ extends to an element of $B^{*}$, hence 
	\begin{equation*}
	F(T, S, A \otimes S) = (A \otimes S) \cap F(T, B, A \otimes B).
	\end{equation*}
It follows that $F(T, S, A \otimes S) = F(T, S, A \otimes B)$, by (\ref{eq F}).
\end{proof}

\begin{cor}\label{cor trans} Let $T \subseteq A$ and $S \subseteq B$ be operator spaces. Suppose that $(A, S, A \otimes B)$ and $(T, S, A \otimes S)$ have the slice map property. Then $(T, S, T \otimes B)$ also has the slice map property.
\end{cor}
\begin{proof} Follows from Lemma~\ref{lem aotimes} and the following commutative diagram of inclusions:
	\begin{equation*}
	\xymatrix{
	F(T, S, A \otimes S) \ar@{^{(}->}[r] & F(T, S, A \otimes B)\\
	T \otimes S \ar@{^{(}->}[r] \ar@{^{(}->}[u]& F(T, S, T \otimes B)  \ar@{^{(}->}[u]\\
 	}.
	\end{equation*}
\end{proof}

\begin{thm}\label{thm cug slice} Let $\G$ be a countable {\em exact} discrete group. If the equality 
	\begin{equation*}
	\Cu(\G)^{\G} \otimes S = (\Cu(\G) \otimes S)^{\G}
	\end{equation*}
holds for all operator spaces $S \subseteq \K(H)$, then $\Cu(\G)^{\G}$ has the slice map property for $\K(H)$.
\end{thm}
\begin{proof} Let $T \coloneqq \Cu(\G)^{\G}$, $A \coloneqq \Cu(\G)$ and $B \coloneqq \K(H)$. By \cite[Theorem 3]{MR1763912}, the algebra $A$ is nuclear and therefore has the OAP (see Definition~\ref{defn oap}). From \cite[Theorem 5.5]{MR1138840} we see that $(A, S, A \otimes B)$ has the slice map property for all operator spaces $S \subseteq B$. By Corollary~\ref{cor cug}, the triple $(T, S, A \otimes S)$ has the slice map property for all operator spaces $S \subseteq B$. Now Corollary~\ref{cor trans} completes the proof.
\end{proof}

\subsection{Operator Approximation Property}\label{sec oap}

Now we relate the slice map property to the approximation property.
\begin{defn}[{cf.\ \cite{MR1060634}}]\label{defn oap} 
A completely bounded map $\theta\colon A \to B$ of operator spaces is said to have the {\em operator approximation property} (OAP) if there is a net of continuous finite-rank maps $F_{n}\colon A \to B$ converging to $\theta$ in the stable point-norm topology: $F_{n} \otimes \id_{\K(H)}(a) \to \theta \otimes \id_{\K(H)}(a)$ for all $a \in A \otimes \K(H)$ as $n \to \infty$.    
We say that an operator space $A$ has the OAP if $\id_{A}\colon A \to A$ has the OAP.
\end{defn}

By \cite[Theorem 5.5]{MR1138840}, a \cast-algebra has the OAP if and only if it has the slice map property for $\K(H)$. Combining with Theorem~\ref{thm cug slice}, we get the following.
\begin{cor}\label{cor cug oap} Let $\G$ be a countable {\em exact} discrete group. If the equality 
	\begin{equation*}
	\Cu(\G)^{\G} \otimes S = (\Cu(\G) \otimes S)^{\G}
	\end{equation*}
holds for all operator spaces $S \subseteq \K(H)$, then $\Cu(\G)^{\G}$ has the OAP.
\qed
\end{cor}

Finally, by a result of Haagerup and Kraus, a countable discrete group $\G$ has the AP if and only if $\Cl(\G)$ has the OAP (cf.\ \cite[Theorem 2.1]{MR1220905}). 
A simple inspection shows that the proof of \cite[Theorem 2.1(c) $\Rightarrow$ (a)]{MR1220905} actually demonstrates the following.

\begin{prop}\label{prop imath} Let $\G$ be a countable discrete group and let $\imath_{\G}\colon \Cl(\G) \to \L(\G)$ denote the inclusion map. If $\imath_{\G}$ has the OAP, then $\G$ has the AP. \qed
\end{prop}

Now the following corollary is immediate. 
\begin{cor}\label{cor cuop ap}
The algebra $\Cl(\G)$ has the OAP if and only if $\Cu(\G)^{\G}$ has the OAP if and only if $\G$ has the AP.
\end{cor}
\begin{proof} If $\Cl(\G)$ or $\Cu(\G)^{\G}$ has the OAP, then so does $\imath_{\G}$ and thus $\G$ has the AP by Proposition~\ref{prop imath}. Conversely, if $\G$ has the AP, then $\Cl(\G)$ has the OAP by \cite[Theorem 2.1]{MR1220905} and $\Cu(\G)^{\G} = \Cl(\G)$ by Theorem~\ref{thm Zac}.
\end{proof}


Now we are ready to prove Theorem~\ref{thm main}.
\begin{proof}[Proof of Theorem~\ref{thm main}] In view of Theorem~\ref{thm Zac}, it suffices to prove (\ref{it Slice}) $\Rightarrow$ (\ref{it AP}). However, this is clear from Corollary~\ref{cor cug oap} and Corollary~\ref{cor cuop ap}.
\end{proof}


\section{Other Approximation Properties}
We end with a theorem which is in agreement with the conjecture that all countable discrete groups have the ITAP. See \cite{MR2391387} for the terminologies used. Compare Corollary~\ref{cor cuop ap}.

\begin{thm}\label{thm app} Let $\G$ be a countable discrete group. The following statements hold.
\begin{enumerate}[(i)]
\item\label{nuclear} The algebra $\Cl(\G)$ is nuclear if and only if $\Cu(\G)^{\G}$ is nuclear if and only if $\G$ is amenable. 
\item\label{CBAP} The algebra $\Cl(\G)$ has the CBAP if and only if $\Cu(\G)^{\G}$ has the CBAP if and only if $\G$ is weakly amenable.
\item\label{SOAP} The algebra $\Cl(\G)$ has the SOAP if and only if $\Cu(\G)^{\G}$ has the SOAP if and only if $\G$ has the AP.
\item\label{exact} The algebra $\Cl(\G)$ is exact if and only if $\Cu(\G)^{\G}$ is exact if and only if $\G$ is exact.
\end{enumerate}
\end{thm}

The proof is based on the following analogue of Proposition~\ref{prop imath}.
\begin{lem}\label{lem i} Let $\G$ be a countable discrete group.
\begin{enumerate}[(i)]
\item\label{i nuclear} The inclusion $\imath_{\G} \colon \Cl(\G) \to \L(\G)$ is nuclear if and only if $\G$ is amenable.
\item\label{i CBAP} The inclusion $\imath_{\G} \colon \Cl(\G) \to \L(\G)$ has the CBAP if and only if $\G$ is weakly amenable.
\item\label{i SOAP} The inclusion $\imath_{\G} \colon \Cl(\G) \to \L(\G)$ has the SOAP if and only if $\G$ has the AP.
\end{enumerate}
\end{lem}
\begin{proof}
\begin{enumerate} 
\item[(\ref{i nuclear})] Follows from the proof of \cite[Theorem 2.6.8(9) $\Leftrightarrow$ (1)]{MR2391387}.
\item[(\ref{i CBAP})] Follows from the proof of \cite[Theorem 12.3.10]{MR2391387}.
\item[(\ref{i SOAP})] Follows from Proposition~\ref{prop imath} and \cite[Theorem 2.1]{MR1220905}.
\end{enumerate}
\end{proof}

\begin{proof}[Proof of Theorem~\ref{thm app}] As remarked in Example~\ref{ex ap}, amenable and weakly amenable groups have the AP, thus have the ITAP by Theorem~\ref{thm Zac}. Combining with Lemma~\ref{lem i}, we get the statements (\ref{nuclear}), (\ref{CBAP}) and (\ref{SOAP}).

Finally, the statement (\ref{exact}) follows from \cite[Theorem 3]{MR1763912}, since exactness passes to subalgebras. 
\end{proof}

\bibliographystyle{amsalpha}
\bibliography{../BibTeX/biblio}

\end{document}